\documentclass[a4paper,10pt]{article}
\usepackage[utf8]{inputenc}
\usepackage[T1]{fontenc}
\usepackage[english]{babel}
\usepackage{amsmath}
\usepackage{amsthm}
\usepackage{amsfonts}
\usepackage{amssymb}
\usepackage{listings}
\usepackage{graphicx}
\usepackage{mathrsfs}
\usepackage{hyperref}
\usepackage{algpseudocode}
\usepackage{algorithmicx}
\usepackage{algorithm}
\usepackage{subfig}
\usepackage{stmaryrd}
\usepackage{bm}
\usepackage{tikz}
\usepackage{braket}
\usepackage{wasysym}
\usepackage{wrapfig}
\usepackage[framemethod=TikZ]{mdframed}
\usepackage{xcolor}
\usepackage{pstricks}
\usepackage{faktor}

\usepackage[top=2.00cm,bottom=3.00cm,left=2.00cm,right=2.00cm]{geometry}

\theoremstyle{definition}

\theoremstyle{plain}
\newtheorem{theorem}{Theorem}
\newtheorem{proposition}{Proposition}

\newtheorem{conjecture}{Conjecture}

\title{A short note on $A_\alpha$-eigenvalues for simple graphs}
\date{}
\author{Giovanni Barbarino\thanks{Departamento de Ciencias Integradas, Facultad de Ciencias Experimentales, Universidad de Huelva, campus El Carmen, Huelva 21007, Spain.  GB is member of the Research Group GNCS (Gruppo Nazionale per il
Calcolo Scientifico) of INdAM (Istituto Nazionale di Alta Matematica). 
Email: giovanni.barbarino@dci.uhu.es. } } 

\begin{document}

\maketitle
\begin{abstract}
  Given a simple graph $G$, its $A_\alpha$ matrix is a convex combination with parameter $\alpha\in [0,1]$ of its adjacency matrix and its degree diagonal matrices.  Here we compare two lower bounds presented in \cite{1} for the spectral radius of $A_\alpha$, and prove that one is better than the other when there are no isolated nodes in $G$. 
\end{abstract}

\section{Introduction}

Let $G$ be a simple graph with $n$ vertices. Let $D$ be the diagonal matrix containing the degrees of the nodes of $G$ on its diagonal, and let $A$ be the adjacency matrix of $G$. For any  $\alpha\in [0,1]$ define the $A_\alpha$ matrix (\cite{2}) associated to $G$ as
\[
A_\alpha = \alpha D + (1-\alpha) A,
\]
where $A_0\equiv A$, $A_1\equiv D$ and $A_{1/2}$ is half the signless Laplacian matrix of $G$. $A_\alpha$ is in particular a symmetric and entry-wise nonnegative matrix, thus its spectral radius $\rho(A)$ coincides with its largest nonnegative eigenvalue $\lambda_1(A)$ and also with its spectral norm $\|A\|$. 

Estimating the eigenvalues of $A_\alpha$ has been of interest to many researchers, see, among the most recent articles, \cite{3,4}. In particular,  its spectral radius $\lambda_1(A_\alpha)$  (see \cite{5})  has been given much importance since its value can be studied to guarantee the presence of specific spanning subgraphs (called factors) in $G$ (\cite{8,7,6}).

Several lower and upper bounds have been formulated throughout the years, linking the spectral radius $\lambda_1(A_\alpha)$ to the parameter $\alpha$ and the values $\Delta$, $\delta$, that are respectively  the maximum degree and the minimum degree of the graph $G$.  
In \cite{2,1,3} many bounds for $\lambda_1(A_\alpha)$ can be found, here we report two of them.

\begin{proposition}
\cite{2}   \label{bound1} If $G$ is a graph with maximum degree $\Delta$ and  $\alpha\in [0,1]$, then 
    \[
    \lambda_1(A_\alpha) \ge \frac 12 \left( \alpha(\Delta+1)+\sqrt{\alpha^2(\Delta+1)^2 + 4\Delta(1-2\alpha)}  \right).
    \]
        If $G$ is connected, the equality holds if and only if $G\cong K_{1,\Delta}$.
\end{proposition}
\begin{theorem}
    \cite{1} \label{bound2} Let $G$ be a graph of order $n\ge 2$, $\Delta$ and $\delta$ be the maximum degree and the minimum degree of $G$, respectively, and $\alpha\in [0,1]$. Then 
       \[
    \lambda_1(A_\alpha) \ge \frac 12 \left( \alpha(\Delta+\delta)+\sqrt{\alpha^2(\Delta-\delta)^2 + 4\Delta(1-\alpha)^2}  \right).
    \]
    If $G$ is connected, the equality holds if and only if $G\cong K_{1,n-1}$.
\end{theorem}

In \cite{1}, it is proved that the bound in Theorem \ref{bound2}  is greater than the one in Proposition \ref{bound1} for $r$-regular graphs (i.e. $\Delta=\delta=r$) . Based on computational tests, it is also empirically observed that the same relation holds for non-regular graphs with $\delta>1$. As a consequence, the following conjecture have been formulated.

\begin{conjecture}
\label{conj}  \cite{1}  Let $G$ be a simple graph, $\Delta$ its maximum degree, $\delta$ its minimum degree and $\alpha\in [0,1]$. Then 
    \[
    \alpha(\Delta + \delta) +\sqrt{\alpha^2(\Delta-\delta)^2 + 4\Delta(1-\alpha)^2}
    \ge 
     \alpha(\Delta + 1) +\sqrt{\alpha^2(\Delta+1)^2 + 4\Delta(1-2\alpha)}.
    \]\end{conjecture}

In the following, we show that the conjecture holds exactly for $\delta\ge 1$, while it is reversed for $\delta=0$. 

\section{Main Results}

First of all, notice that 
\begin{equation}
    \label{eq:sec_rad}
\alpha^2(\Delta+1)^2 + 4\Delta(1-2\alpha) = \alpha^2(\Delta-1)^2 + 4\Delta(1-\alpha)^2 \ge 0
\end{equation}
so the argument of the second square root in the bound of Proposition \ref{bound1} is never negative, and all formulas are well defined. Here we prove that for $\delta \ge 1$, Conjecture \ref{conj} holds.

\begin{proposition}
\label{prop1}    Suppose $\Delta \ge \delta \ge 1$ and $\alpha\ge 0$. Then 
        \[
    \alpha(\Delta + \delta) +\sqrt{\alpha^2(\Delta-\delta)^2 + 4\Delta(1-\alpha)^2}
    \ge 
     \alpha(\Delta + 1) +\sqrt{\alpha^2(\Delta+1)^2 + 4\Delta(1-2\alpha)}
    \]
    and the equality holds if and only if $\delta = 1$ or $\alpha = 0$ or $\alpha = 1$. 
\end{proposition}
\begin{proof}
    After a first simplification and applying \eqref{eq:sec_rad}, the thesis becomes
            \[
    \alpha(\delta-1) +\sqrt{\alpha^2(\Delta-\delta)^2 + 4\Delta(1-\alpha)^2}
    \ge 
    \sqrt{\alpha^2(\Delta-1)^2 + 4\Delta(1-\alpha)^2}.
    \]
    Since $\alpha\ge 0$ and $\delta\ge 1$, then  both sides are nonnegative and we can square everything. In order to carry on the computation, let us call
    \[
    x :=\alpha^2(\Delta-\delta)^2 + 4\Delta(1-\alpha)^2, \qquad y:= \alpha^2(\Delta-1)^2 + 4\Delta(1-\alpha)^2
    \]
    so that the thesis is equivalent to 
    \[
    \alpha^2(\delta-1)^2 + 2\alpha(\delta-1)\sqrt x + x \ge y
    \]
    or  
    \[
       2\alpha(\delta-1)\sqrt x  \ge y - x -  \alpha^2(\delta-1)^2.
    \]
    Now 
 \begin{align*}
      y - x -  \alpha^2(\delta-1)^2 &= 
      \alpha^2(\Delta-1)^2 + 4\Delta(1-\alpha)^2 - \alpha^2(\Delta-\delta)^2 - 4\Delta(1-\alpha)^2 - \alpha^2(\delta-1)^2\\
      &= \alpha^2\left[ 
      (\Delta-\delta + \delta - 1)^2  - (\Delta-\delta)^2 -(\delta-1)^2
      \right]\\
      &= 2\alpha^2(\Delta-\delta)(\delta-1),\\
 \end{align*} 
 and substituting we get
 \[
  2\alpha(\delta-1)\sqrt x  \ge 2\alpha^2(\Delta-\delta)(\delta-1).
 \]
 From the observation that  $\sqrt x \ge \alpha(\Delta-\delta)$, we conclude that the thesis holds. Moreover the two sides are equal only when $\alpha(\delta-1) =0$, giving us the conditions $\alpha=0$ and $\delta=1$, or when $\sqrt x = \alpha(\Delta-\delta)$ that happens only for $4\Delta(1-\alpha)^2 = 0$, i.e. $\alpha = 1$. 
\end{proof}

Notice that the result agrees with the observation made in \cite{1}, stating that  the two bounds are equal for $\delta=1$. 
Since Proposition \ref{prop1} requires $\delta\ge 1$, we can study what happens for $\delta = 0$ (i.e. when there are isolated nodes). Notice that in this case the graph $G$ is disconnected. In this case we can prove that Conjecture \ref{conj}
 is false and actually  the bound in Theorem \ref{bound2}  is always smaller or equal than the one in Proposition \ref{bound1}.
 
\begin{proposition}
 \label{prop2}   Suppose $\Delta \ge \delta =0$ and $\alpha\ge 0$. Then 
        \[
    \alpha(\Delta + \delta) +\sqrt{\alpha^2(\Delta-\delta)^2 + 4\Delta(1-\alpha)^2}
    \le 
     \alpha(\Delta + 1) +\sqrt{\alpha^2(\Delta+1)^2 + 4\Delta(1-2\alpha)}
    \]
    and the equality holds if and only if $\alpha = 0$, or $\alpha = 1$ and $\Delta  \ge  1$. 
\end{proposition}
\begin{proof}
   By imposing $\delta =0$ and applying \eqref{eq:sec_rad}, the thesis becomes    \[
     \sqrt{\alpha^2\Delta^2 + 4\Delta(1-\alpha)^2}
    \le 
     \alpha +\sqrt{\alpha^2(\Delta-1)^2 + 4\Delta(1-\alpha)^2}.
    \]
Since $\alpha\ge 0$, we can square both sides and obtain the following equivalent relations:
\[
     \alpha^2\Delta^2 + 4\Delta(1-\alpha)^2
    \le 
    \alpha^2 + 
     2\alpha \sqrt{\alpha^2(\Delta-1)^2 + 4\Delta(1-\alpha)^2} + 
     \alpha^2(\Delta-1)^2 + 4\Delta(1-\alpha)^2
    \]
\[
  \alpha^2
  [ \Delta^2  - 1 -  (\Delta-1)^2]
    \le 
      2\alpha \sqrt{\alpha^2(\Delta-1)^2 + 4\Delta(1-\alpha)^2}
\]
\[
  2\alpha^2  (\Delta-1)
    \le 
      2\alpha \sqrt{\alpha^2(\Delta-1)^2 + 4\Delta(1-\alpha)^2}
\]
The last relation is always true since $ 4\Delta(1-\alpha)^2\ge 0$. 
For $\alpha=0$ the two sides are equal, otherwise we can simplify and obtain that
\[
 \alpha (\Delta-1)
    = 
      \sqrt{\alpha^2(\Delta-1)^2 + 4\Delta(1-\alpha)^2}.
\]
The last relation holds only when $4\Delta(1-\alpha)^2 = 0$ and $\Delta \ge 1$ are both satisfied, so we need $\alpha = 1$ and $\Delta \ge 1$. 
\end{proof}

We can then put together Proposition \ref{prop1} and \ref{prop2} to find all the cases when the bound in Theorem \ref{bound2}  is greater than the one in Proposition \ref{bound1}, and when they are equal.

 \begin{theorem}
   Let $f(\delta,\Delta,\alpha)$ be the bound in Theorem \ref{bound2}  and let $g(\Delta,\alpha)$ be the bound in  Proposition \ref{bound1}, where $\alpha\ge 0$ and $\Delta\ge\delta$ are nonnegative integers. We have that
\begin{itemize}
\item $f(\delta,\Delta,\alpha)> g(\Delta,\alpha) \iff \delta\ge2,\ \alpha\ne 0,1$,
\item $f(\delta,\Delta,\alpha)= g(\Delta,\alpha)$ if and only if one of the following hold
\begin{itemize}
    \item $\alpha = 0$,
    \item $\alpha= 1$ and $\Delta\ne 0$,
    \item $\delta=1$,
\end{itemize}
\item$f(\delta,\Delta,\alpha)< g(\Delta,\alpha)$  if and only if one of the following hold
\begin{itemize}
    \item $\delta=0$ and $\alpha\ne 0,1$,
    \item $\Delta=0$ and $\alpha\ne 0$.
\end{itemize}
     \end{itemize}
\end{theorem}
\begin{proof}
  Let us call $f=f(\delta,\Delta,\alpha)$ and $g= g(\Delta,\alpha)$ for brevity.  Notice that Proposition \ref{prop1} covers the case $\delta\ge 1$ and Proposition \ref{prop2} covers the case $\delta=0$, so together we have a complete description of the relations between $f$ and $g$. 

  In particular, by Proposition \ref{prop2} we have $\delta = 0 \implies f\le g$, so  $ f> g \implies \delta\ge 1$ and 
   by Proposition \ref{prop1}  in this case $f\ge g$ with equality only in the cases $\delta=1$ or $\alpha =0$ or $\alpha = 1$. 
   We conclude that 
$$     f>g\iff 
(\delta\ne 1,\ \alpha \ne 0,1 )
    \wedge (\delta\ge 1)\iff \delta \ge 2,\ \alpha \ne 0,1 .$$
    
The equality conditions given in Proposition \ref{prop1} and Proposition \ref{prop2} are 
\begin{enumerate}
    \item $\delta = 1$ (Proposition \ref{prop1}),
    \item $\delta \ge 1$, $\alpha = 0$ (Proposition \ref{prop1}),
    \item $\delta \ge 1$, $\alpha = 1$ (Proposition \ref{prop1}),
    \item $\delta =0$, $\alpha = 0$ (Proposition \ref{prop2}),
    \item $\delta =0$, $\alpha = 1$, $\Delta \ne 0$ (Proposition \ref{prop2}).
\end{enumerate}
Conditions number 2. and 4. can be joined into  $\alpha=0$. 
Conditions number 3. and 5. can be joined into $\alpha=1$, $\Delta \ne 0$. 
This proves that $f=g$ if and only if one of the conditions in the thesis is met. 

Notice now that if $\delta\ge 1$ then $f\ge g$ by Proposition \ref{prop1}, so $f<g$ if and only if $\delta=0$ and the equality conditions 4. and 5. above are not met, i.e.
$$     f<g\iff 
 (\delta=0) \wedge
(\alpha \ne 0 )\wedge 
[(\Delta = 0) \vee (\alpha \ne 1)] 
   \iff 
    (\alpha \ne 0, \Delta =0 )\vee (\delta= 0, \alpha \ne 0, 1 ).$$
\end{proof}

%% to use mendeley-imported bib
%\printbibliography

\end{document}